\documentclass{article}
\usepackage{amsthm}
\usepackage{amssymb}
\usepackage{amsmath}
\usepackage{epsfig}
\usepackage{colonequals}
\usepackage{enumerate}
\usepackage{hyperref}
\usepackage{xcolor} 

\newcommand{\TT}{{\mathcal T}}
\newcommand\ZZ{\mathbb Z}
\let\phi\varphi
\newtheorem{theorem}{Theorem}[section]
\newtheorem{corollary}[theorem]{Corollary}
\newtheorem{lemma}[theorem]{Lemma}

\newtheorem{conjecture}[theorem]{Conjecture}

\begin{document}
\title{Exponentially many nowhere-zero $\ZZ_3$-, $\ZZ_4$-, and $\ZZ_6$-flows%
  \thanks{Preliminary version of this work appears as extended abstract from Euro\-comb 2017, published as~\cite{DMS-ea}.}%
}
\author{%
     Zden\v{e}k Dvo\v{r}\'ak\thanks{Computer Science Institute (CSI) of Charles University,
           Malostransk{\'e} n{\'a}m{\v e}st{\'\i} 25, 118 00 Prague,
           Czech Republic. E-mail: \protect\href{mailto:rakdver@iuuk.mff.cuni.cz}{\protect\nolinkurl{rakdver@iuuk.mff.cuni.cz}}.
           Supported by project 14-19503S (Graph coloring and structure) of Czech Science Foundation.}\and
     Bojan Mohar\thanks{Department of Mathematics, Simon Fraser University, Burnaby, B.C. V5A 1S6.
           E-mail: {\tt mohar@sfu.ca}.  Supported in part by an NSERC Discovery Grant (Canada),
           by the Canada Research Chairs program, and by a Research Grant of ARRS (Slovenia).}\and
     Robert \v{S}\'amal\thanks{Computer Science Institute (CSI) of Charles University,
           Malostransk{\'e} n{\'a}m{\v e}st{\'\i} 25, 118 00 Prague,
           Czech Republic. E-mail: \protect\href{mailto:samal@iuuk.mff.cuni.cz}{\protect\nolinkurl{samal@iuuk.mff.cuni.cz}}.
		   Partially supported by grant 16-19910S of the Czech Science Foundation.
		   Partially supported by grant LL1201 ERC CZ of the Czech Ministry of Education, Youth and Sports.}
}
\date{\today}

\maketitle

\begin{abstract}
  We prove that, in several settings, a graph has exponentially many nowhere-zero flows.
  These results may be seen as a counting alternative to the well-known proofs of existence of $\ZZ_3$-, $\ZZ_4$-, and $\ZZ_6$-flows.
  In the dual setting, proving exponential number of 3-colorings of planar triangle-free graphs is a related open question due to Thomassen.
\end{abstract}


\section{Introduction}

Our graphs may have loops and parallel edges.
Let $G$ be a graph with a given orientation of its edges and let~$\Gamma$ be an abelian group.
A mapping $\phi: E(G) \to \Gamma$ is called a \emph{flow} if for every vertex~$v \in V(G)$
it satisfies the Kirchhoff's law
$$
   \sum_{e \in \delta^+(v)} \phi(e) = \sum_{e \in \delta^-(v)} \phi(e)
$$
(here $\delta^+(v)$, $\delta^-(v)$ denote the set of edges directed away from/towards~$v$).
We say that $\phi$ is a $\Gamma$-flow to express the group we are using.
Further, we say that $\phi$ is \emph{nowhere-zero} if $\phi(e) \ne 0$ for every $e \in E(G)$ and we say
that $\phi$ is a \emph{$k$-flow} if $\Gamma = \ZZ$ and $|\phi(e)| < k$ for every edge~$e$.
Note that, while we need some orientation to define a flow, the orientation itself is irrelevant: if we
reverse an arc and change the sign of its flow-value, the Kirchhoff's law still holds true.

The study of nowhere-zero flows was initiated by Tutte~\cite{tutte1, tutte2}. The main motivation was the
following duality theorem.

\begin{theorem}\label{thm:duality}
Let $G$ be a plane graph and let $G^*$ be the dual of~$G$. Then
$G$~has a nowhere-zero $\ZZ_k$-flow if and only if $G^*$ is $k$-colorable.
If $f_k$ is the number of nowhere-zero $\ZZ_k$-flows on~$G$ and $c_k$ the number of $k$-colorings of~$G^*$, then
$c_k = k\cdot f_k$.
\end{theorem}

As a consequence, the questions about chromatic number of planar graphs, that were always at the core of graph theory,
can be studied in a new setting. This line of thought had lead to the following conjectures, motivated by the
Gr\"otzsch Theorem and by the Four Color Theorem (still a conjecture then).

\begin{conjecture}[Tutte~\cite{tutte1, tutte2}]\label{c:flowconj}
\
\begin{itemize}
  \item Every 4-edge-connected graph has a nowhere-zero 3-flow.
  \item Every 2-edge-connected graph with no Petersen minor has a nowhere-zero 4-flow.
  \item Every 2-edge-connected graph has a nowhere-zero 5-flow.
\end{itemize}
\end{conjecture}

Tutte proved that, surprisingly, the concept of $\Gamma$-flows does not
depend on the structure of~$\Gamma$, but only on its size.

\begin{lemma}\label{l:indgroup}
  Let $G$ be a graph, $k \ge 2$ an integer and $\Gamma$ any abelian group of
  order~$k$. Then the following are equivalent:
\begin{itemize}
  \item $G$ has a nowhere-zero $k$-flow.
  \item $G$ has a nowhere-zero $\Gamma$-flow.
\end{itemize}
\end{lemma}

Tutte also proved that the number of nowhere-zero $\Gamma$-flows on~$G$ is equal to
$p_G(|\Gamma|)$ for some polynomial~$p_G$ depending on~$G$. Thus, if the graph is fixed and
we enlarge the group, the number of nowhere-zero flows grows polynomially. In this paper, we will
show that when we do the opposite---keep the same group and enlarge the graph---then
in many cases the number of nowhere-zero flows grows exponentially.

As a consequence of Lemma~\ref{l:indgroup}, when proving the existence of a nowhere-zero $k$-flow, we may instead
work with $\ZZ_k$-flows; for $k=4$, the group $\ZZ_2^2$ is frequently useful.
Tutte's 3-flow, 4-flow, and 5-flow conjectures (Conjecture~\ref{c:flowconj})
are still open. They inspired a lot of partial results, though. For our
work, the following are most relevant:

\begin{theorem}[Seymour~\cite{Seymour-6flow}] \label{thm:S6flow}
Every 2-edge-connected graph has a nowhere-zero $\ZZ_2 \times \ZZ_3$-flow.
\end{theorem}

\begin{theorem}[Jaeger~\cite{Jaeger-4flow}] \label{thm:J4flow}
Every 4-edge-connected graph has a nowhere-zero $\ZZ_2 \times \ZZ_2$-flow.
\end{theorem}

\begin{theorem}[Lov\'asz, Thomassen, Wu, and Zhang~\cite{ltwz}] \label{thm:3flow}
  \ \\
Every 6-edge-con\-nected graph has a nowhere-zero $\ZZ_3$-flow.
\end{theorem}

In this paper we revisit the above three theorems with a new point of view: counting.
We were motivated to this by several recent results and conjectures.
We start with the celebrated proof of Lov\'asz--Plummer conjecture by Esperet et al.~\cite{EKKKN-expmatch}:
every bridgeless cubic graph has exponentially many perfect matchings.
More directly related is the result of Thomassen: every planar graph has
exponentially many (list) 5-colorings. Thomassen~\cite{Thom-exp3} also asked the same question for 3-colorings
of triangle-free planar graphs. He gave a subexponential bound that was later improved by Asadi et al.~\cite{ADPT-exp3}.
However, the conjecture stays open.
By duality, these results and conjecture can be equivalently stated for the number of nowhere-zero
$\ZZ_5$-flows (or $\ZZ_3$-flows) of planar (4-edge-connected) graphs.
Perhaps, in the spirit of Conjecture~\ref{c:flowconj}, planarity is not required.
We do not attempt to prove Thomassen's conjecture here, much the less its strengthening; one particular issue is that
proving existence of just one $\ZZ_3$-flow for every 4-edge-connected graph is Tutte's
4-flow conjecture.
Still, we hope that our results may serve as an inspiration for others working on this and related conjectures.

In the rest, we will let $n$ be the number of vertices and $m$ the number of edges of
the studied graph.

We count nowhere-zero $\ZZ_2\times \ZZ_3$-flows in Section~\ref{sec:6flow}. Obviously,
2-edge-connectivity is not sufficient, as a long cycle still has just five nowhere-zero
$\ZZ_2\times\ZZ_3$-flows. We provide two bounds for the number of nowhere-zero flows: $2^{n/7}$ for 3-edge-connected graphs
and $2^{2(m-n)/9}$ for 2-edge-connected ones.
An $n$-cycle with $m-n$ edges doubled has $5 \cdot 2^{2(m-n)}$ nowhere-zero $\ZZ_2\times\ZZ_3$-flows---thus, the
second result is optimal, up to possibly increasing the $2/9$-factor in the exponent.
The $n$-vertex cycle with all edges doubled ($m=2n$) is $4$-edge-connected,
giving an upper bound~$5\cdot 4^n$, showing near optimality of the first result.
Note, however, that we did not attempt to provide the best possible bounds.

In Section~\ref{sec:4flows}, we count nowhere-zero $\ZZ_2^2$-flows, proving the existence of
$2^{n/250}$ of them in any 4-edge-connected graph (any snark shows that 3-edge-connectivity is not
sufficient even for existence of one nowhere-zero $\ZZ_2^2$-flow).
As in Jaeger's proof of Theorem~\ref{thm:J4flow} we need 4-edge-connectivity only to get the existence
of two edge-disjoint spanning trees. To make this explicit, we prove the existence of many flows just
assuming the existence of two edge-disjoint spanning trees.

Finally, in Section~\ref{sec:3flows} we count nowhere-zero $\ZZ_3$-flows.
Unlike the previous sections, where we prove everything from first principles, here we
rely on the result of L.\,M.\,Lov\'asz et al.~\cite{ltwz}, which in turn is based on the 
of Thomassen~\cite{Thomassen-3flow}. 
We prove existence of $2^{(n-2)/12}$ nowhere-zero $\ZZ_3$-flows.

We finish the introduction by recalling a tool that we will use frequently.
\emph{Lifting} a pair of edges $e_1$ and $e_2$ incident with the same vertex $v$ in $G$ means replacing them by
one edge joining their other end-vertices (and deleting $v$ if there are no edges incident with $v$ left).  A pair of edges $e_1$ and $e_2$ incident with the same vertex $v$ in a graph $G$
is \emph{splittable} if for all $s,t\in V(G-v)$, the graph obtained from $G$ by lifting $e_1$ and $e_2$ contains
the same maximum number of pairwise edge-disjoint paths from $s$ to $t$ as $G$ does.  Mader~\cite{mader} proved the following
result, frequently called splitting lemma.

\begin{lemma}[Mader~\cite{mader}]\label{lemma-mader}
If a vertex $v$ has degree either $2$ or at least $4$ and $v$ is not incident with a cut-edge, then $v$~is incident with a splittable pair of edges.
\end{lemma}

We will also use the following easy corollary.

\begin{corollary}\label{cor-mader2}
Let $G$ be a $k$-edge-connected graph (for $k \ge 2$) and let $v$ be a vertex of~$G$ of degree at least~$k+2$. Then there is pair of splittable edges incident with~$v$
such that after lifting them, the resulting graph~$G'$ is again $k$-edge-connected.
\end{corollary}

\begin{proof}
  Suppose a nonempty $X \subset V(G')$ has less than $k$ edges leaving it in~$G'$. If there is $s \in X$, $t \in V(G')\setminus X$ such that $s \ne v \ne t$, then
  we get a contradiction with the definition of a splittable pair. Thus we may assume that $X = \{v\}$. But $\deg_{G'} v \ge k+2-2 = k$, which completes the proof.
\end{proof}

\section{$\ZZ_2\times \ZZ_3$-flows}\label{sec:6flow}

Our starting point is Seymour's proof of Theorem~\ref{thm:S6flow} from \cite{Seymour-6flow}. With careful counting added, this gives our
Lemma~\ref{lemma-count}. However, more work is needed if the graph has few edges.
Moreover, Seymour starts by reducing to a cubic graph, which simplifies the proof.
We cannot do this here, as the reductions change the size of the graph---thus we need to control how the number of flows changes
by the reduction. We will have vertices of arbitrary degree in the considered graph, and to handle this, we need the following definitions.

A \emph{$(u,v)$-chain} (or a \emph{chain} from $u$ to $v$) is a graph obtained from a path with ends $u$ and $v$ by doubling all edges
and possibly subdividing some of them.  It is easy to see that if a graph $G$ contains two edge-disjoint paths from $u$ to $v$
and $G$ is subgraph-minimal subject to this property, then $G$ is a $(u,v)$-chain.  A single vertex~$u$ is also considered to be a chain, namely a $(u,u)$-chain.

An \emph{anchored chain cover} of a graph $G$ consists of vertex-disjoint chains $C_1$, \ldots, $C_k$, where each $C_i$ is joining
vertices $u_i$ and $v_i$ ($i=1,\ldots, k$), such that $V(G)=V(C_1\cup \cdots\cup C_k)$, $C_1$ is a cycle, and for each $i=2,\ldots,k$,
$G$ contains two distinct edges $u_ix_i$ and $v_iy_i$ with $x_i,y_i\in V(C_1\cup\cdots\cup C_{i-1})$.
The edges $u_ix_i$ and $v_iy_i$ are called \emph{anchors} of $C_i$ and are considered being part of the anchored chain cover.
The edges of $G$ not belonging to the chains
of the cover and not equal to the anchors are called \emph{external}.

\begin{lemma}\label{lemma-cover}
Every $3$-edge-connected graph $G$ contains an anchored chain cover.
\end{lemma}
\begin{proof}
We construct the cover inductively as follows. We let $C_1$ be an arbitrary cycle in $G$.
Suppose that we have already found chains $C_1$, \ldots, $C_{i-1}$ joining vertices $u_j$ and $v_j$ for $j=1,\ldots, i-1$, such that for $j=2,\ldots,i-1$,
$G$ contains two distinct edges $u_jx_j$ and $v_jy_i$ with $x_j,y_j\in V(C_1\cup\cdots\cup C_{j-1})$.
If $V(G)=V(C_1\cup \cdots\cup C_{i-1})$, this forms an anchored chain cover of $G$.

Otherwise, let $Q$ be a leaf $2$-edge-connected component of $G-V(C_1\cup \cdots\cup C_{i-1})$.  Since $G$ is $3$-edge-connected,
there exist two edges $u_ix_i$ and $v_iy_i$ with $u_i,v_i\in V(Q)$ and $x_i,y_i\in V(C_1\cup \cdots\cup C_{i-1})$.
Let $C_i$ be a minimal subgraph of $Q$ that contains two edge-disjoint paths from $u_i$ to $v_i$.
Then $C_i$ is a $(u_i,v_i)$-chain; and we repeat this procedure until an anchored chain cover is found.
\end{proof}

\begin{lemma}\label{lemma-count}
Let $C_1$, \ldots, $C_k$ be an anchored chain cover of a graph $G$ and let $A$ be the set of its anchors.
Let $p$ denote the number of cycles in $C_1\cup \cdots\cup C_k$, and let $X$ denote the set of external edges of $G$.
Let $A'$ be a maximal subset of~$A$ such that each vertex of $G$ is incident with an even number of edges of $A'$.
Then $G$~has at least $2^{|X|}3^{p+|A'|/2}$ nowhere-zero $\ZZ_2\times \ZZ_3$-flows.  Furthermore, $G$ has a nowhere-zero $\ZZ_2\times \ZZ_3$-flow
whose $\ZZ_3$ part is only zero on the edges of $C_1\cup \cdots\cup C_k$, and the number of such zero edges is
at most $|E(C_1\cup \cdots\cup C_k)|/3$.
\end{lemma}
\begin{proof}
Let $Y=E(C_1\cup\cdots\cup C_k)\cup A'$.
Let us fix a $\ZZ_2$-flow in $G$ by assigning $1$ to all edges of $Y$ and $0$ to all other edges.
To construct a nowhere-zero $\ZZ_2\times \ZZ_3$-flow in $G$, it suffices to find a $\ZZ_3$-flow in $G$ where all edges not contained
in $Y$ are non-zero.  For the purpose of specifying the $\ZZ_3$-flow, let us direct
all edges in each cycle of $C_1\cup \cdots\cup C_k$ in the same direction, for $i=2,\ldots, k$, direct one of the anchors of $C_i$
towards its end in $C_i$ and the other anchor away from it, and direct all other edges arbitrarily.  Let $\vec{H}$ be the
directed subgraph of $G$ formed by the chains of the cover and all their anchors.

We start with a zero $\ZZ_3$-flow on $G$.  For each external edge $(u,v)\in X$, note that $\vec{H}$ contains a directed path $P$ from $v$ to $u$.
Set the flow of $(u,v)$ to $1$ or $2$ arbitrarily, and add the same amount to the flow on all edges of the path $P$.
Next, for $i=k,k-1,\ldots, 2$, note that $\vec{H}-V(C_{i+1}\cup\cdots\cup C_k)$ contains a directed cycle $K$ containing both anchors
of $C_i$.  For any $q\in \ZZ_3$ different from the current amounts of flow on the anchors of $C_i$ not belonging to $A'$, subtract $q$
from the amounts of flow on the edges of $K$ (so that the flow on the anchors of $C_i$ not belonging to $A'$ becomes non-zero).
If none of the anchors is in $A'$, there may be unique possibility for $q$; if $A'$ contains one of the anchors, there are two choices for $q$; if both anchors are in $A'$, there are three choices.
All together we have at least $3^{|A'|/2}$ different possibilities.
Finally, for each cycle in $C_1\cup\cdots\cup C_k$, add $0$, $1$, or $2$ to the flow on its edges. There are $3^p$ choices we can make this way.

Clearly, regardless of the amounts of flow assigned to the edges of $X$ and added to the cycles of $C_1\cup\cdots\cup C_k$,
the resulting $\ZZ_3$-flow is non-zero on all edges not belonging to the chains of the cover, and thus together with the fixed $\ZZ_2$-flow,
we obtain a nowhere-zero $\ZZ_2\times \ZZ_3$-flow.  The number of different flows obtained this way is at least $2^{|X|}3^p3^{|A'|/2}$.

Furthermore, we can choose the $\ZZ_3$-flow on each edge of $X$ to be $1$ and on all anchor edges to be non-zero,
and choose the amount added to the flow on each cycle of $C_1\cup\cdots\cup C_k$ so that at most $1/3$ of its edges are
assigned value $0$.  This gives us a nowhere-zero $\ZZ_2\times \ZZ_3$-flow
whose $\ZZ_3$ part is only zero on the edges of $C_1\cup \cdots\cup C_k$, and the number of such zero edges is
at most $|E(C_1\cup \cdots\cup C_k)|/3$.
\end{proof}

\begin{lemma}\label{lemma-dense}
If $G$ is a $3$-edge-connected graph with $n$ vertices and $m$ edges, then $G$~has at least $2^{m-3n/2}$ nowhere-zero $\ZZ_2\times \ZZ_3$-flows.
\end{lemma}

\begin{proof}
Let $C_1$, \ldots, $C_k$ be an anchored chain cover of $G$ that exists by Lemma~\ref{lemma-cover}.
Let $A$, $X$, $A'$ and~$p$ be as in Lemma~\ref{lemma-count}.

Note that $m=|E(G)|=|X|+(n+p-k)+2(k-1)=|X|+n+p+k-2$.  Also, $A\setminus A'$ forms an acyclic subgraph of $G$, and thus
$|A\setminus A'|\le n-1$.  Since $|A|=2(k-1)$, we have $|A'|\ge 2k-n-1$ and $k\le \frac{|A'|+n+1}{2}$.
Therefore, $m\le \frac{3}{2}n+|X|+p+|A'|/2$, and thus $|X|+p+|A'|/2\ge m-3n/2$.  The claimed lower bound on the number of
nowhere-zero $\ZZ_2\times \ZZ_3$-flows then follows from Lemma~\ref{lemma-count}.
\end{proof}

\begin{lemma}\label{lemma-cubic}
If $G$ is a $3$-edge-connected cubic graph with $n$ vertices, then $G$~has at least $2^{n/5}$ nowhere-zero $\ZZ_2\times \ZZ_3$-flows.
\end{lemma}

\begin{proof}
Again, let $C_1$, \ldots, $C_k$ be an anchored chain cover of $G$ guaranteed by Lemma~\ref{lemma-cover}, let $A$
be the set of its anchors, and let $X$ be the set of external edges.
Since $G$ is cubic, each chain of the cover is either a cycle or a single vertex; let $p$ denote the number
of cycles in $C_1\cup \cdots\cup C_k$.  Note that $3n/2=|E(G)|=|X|+n+p+k-2$, and thus
$|X|+p\ge n/2-k$. By Lemma~\ref{lemma-count}, $G$ has at least $2^{|X|+p}\ge 2^{\max(p,n/2-k)}$ nowhere-zero $\ZZ_2\times \ZZ_3$-flows.

It is possible, though, that both $p$ and~$n/2-k$ are small; the extreme case being a union of a cycle and a cubic tree, where
$C_j$ is a single vertex for every $j>1$. In such a case we will use the second part of Lemma~\ref{lemma-count} with the goal
to use the edges that have nonzero $\ZZ_3$-part of the flow to modify the $\ZZ_2$-part. To this end, we first need more information
about the structure of~$G$.

Let $K$ be the union of vertex sets of the cycles among $C_1,\dots,C_k$, and let $J=V(G)\setminus K$. Since $G$ is cubic, $G[J]$ is a forest---otherwise,
let $u$ be the vertex of a cycle of $G[J]$ that belongs to chain $C_j$ for the smallest possible~$j$.
As $u \in J$, we have $j>1$ and thus $u$ is incident with two edges of the cycle plus two anchors, a contradiction.

Let $H$ denote the subgraph of $G$ formed by
all edges incident with $J$ and all external edges with both ends in $K$, and let $q$ be the number of components of $H$.
If $T$ is a component of $G[J]$, then
$G$ contains $3|V(T)|-2(|V(T)|-1)=|V(T)|+2$ edges with one end in $T$ and the other end in $K$.
Note that the ends of these edges in $K$ are pairwise distinct, since $G$ is cubic and $K$~is a collection of cycles.
Similarly, a component of $H$ consisting of an external edge with both ends in $K$ contains no vertices of $J$ and two vertices
of $K$, distinct from those of other components of $H$.  Furthermore, each vertex of $K$ also belongs to $H$, since $K$ is $2$-regular
and $G$ is cubic.
It follows that $|K|=|J|+2q=k-p+2q$.  Since $|K|=|E(C_1\cup\cdots\cup C_k)|=n+p-k$, we have $q=n/2+p-k$.

By Lemma~\ref{lemma-count}, $G$ has a nowhere-zero $\ZZ_2\times \ZZ_3$-flow $f$
whose $\ZZ_3$-part is only zero on the edges of $C_1\cup \cdots\cup C_k$, and the number of such zero edges is
at most $|E(C_1\cup \cdots\cup C_k)|/3=(n-k+p)/3$.  Let $W$ be the set of those edges of $C_1\cup \cdots\cup C_k$,
whose $\ZZ_3$-part of the flow $f$ is non-zero. We have $|W|\ge \frac{2}{3}(n-k+p)$.
Let $W'$ be a maximal subset of $W$ such that $H+W'$ contains no cycle.  Since each edge in $W$ connects two vertices in $H$ and $H$ has $q$ components, $|W'|\le q-1$
and each edge of $W\setminus W'$ has both ends in the same component of $H+W'$.  Now, for any subset $W''$ of $W\setminus W'$,
we can for all $e\in W''$ add $1$ to the $\ZZ_2$-part of the flow $f$ on the edges of the unique cycle in $H+W'+e$.
In this way, we obtain $2^{|W\setminus W'|}\ge 2^{\frac{2}{3}(n-k+p)-q}=2^{(n+2k-2p)/6}$ different nowhere-zero $\ZZ_2\times \ZZ_3$-flows
in~$G$.

In conclusion, $G$ has at least $2^t$ nowhere-zero $\ZZ_2\times \ZZ_3$-flows, where
$t = \max\{p,n/2-k,(n+2k-2p)/6\}$.  Since
\begin{eqnarray*}
   t &=& \max\{p,n/2-k,(n+2k-2p)/6\}\\
     &\ge& \frac{1}{5}(p + (n/2-k) + 3(n+2k-2p)/6) = n/5,
\end{eqnarray*}
the claimed lower bound follows.
\end{proof}

\begin{theorem}\label{thm-manyz6}
If $G$ is a $3$-edge-connected graph with $n$ vertices, then $G$ has at least $2^{n/7}$ nowhere-zero $\ZZ_2\times \ZZ_3$-flows.
\end{theorem}

\begin{proof}
By repeated application of Corollary~\ref{cor-mader2}, we can assume that $G$ has only vertices of degree $3$ and $4$.
(Note that we may also split vertices of degree~4 and preserve connectivity, but this would change the number of vertices.)
Let $n_4$ denote the number of vertices of $G$ of degree $4$.  Then $G$ has $3n/2+n_4/2$ edges.
Let $G'$ be a $3$-edge-connected cubic graph with $n-n_4$ vertices obtained from~$G$ by lifting edges at vertices of degree $4$
using Lemma~\ref{lemma-mader}.  Note that distinct nowhere-zero $\ZZ_2\times \ZZ_3$-flows in $G'$ correspond to distinct nowhere-zero
$\ZZ_2\times \ZZ_3$-flows in $G$.  By Lemmas~\ref{lemma-dense} and \ref{lemma-cubic}, it follows that $G$ has
at least
$$
  2^{\max(n_4/2,(n-n_4)/5)}\ge 2^{n/7}
$$
nowhere-zero $\ZZ_2\times \ZZ_3$-flows.
\end{proof}

We cannot relax the assumption that the graph is $3$-edge-connected in Theorem~\ref{thm-manyz6}, since subdividing edges does not
change the number of nowhere-zero flows.  However, the following holds.

\begin{corollary}\label{cor-over}
If $G$ is a $2$-edge-connected graph with $n$ vertices and $m$ edges, then $G$ has at least $2^{2(m-n)/9}$ nowhere-zero $\ZZ_2\times \ZZ_3$-flows.
\end{corollary}
\begin{proof}
We prove the claim by induction on the number of vertices of $G$.  If an edge $e$ of $G$ is contained in a $2$-edge-cut, then
the graph obtained from $G$ by contracting $e$ has the same number of nowhere-zero $\ZZ_2\times \ZZ_3$-flows as $G$, and the claim
follows from the induction hypothesis.

Otherwise, $G$ is $3$-edge-connected, and by Lemma~\ref{lemma-dense} and Theorem~\ref{thm-manyz6}, $G$~has at least
$$
  2^{\max(m-3n/2,n/7)}\ge 2^{2(m-n)/9}
$$
nowhere-zero $\ZZ_2\times \ZZ_3$-flows.
\end{proof}

\section{$\ZZ_2^2$-flows}\label{sec:4flows}

Jaeger proved that a graph with two edge-disjoint spanning trees has a nowhere-zero
$\ZZ_2^2$-flow. His proof is our point of departure in Lemma~\ref{lemma-noncannon}. Then we combine this with
the possibility to find many different pairs of edge-disjoint spanning trees to get our main result in this section,
Theorem~\ref{thm-4flows}.

In this section we will only consider $\ZZ_2$- and $\ZZ_2 \times \ZZ_2$-flows. Since $\ZZ_2$ and $\ZZ_2 \times \ZZ_2$ have only elements of order 2, there is no need of specifying orientations of the edges.

Let $T$ be a spanning tree of a graph~$G$, and let $f':E(G)\setminus E(T)\to \ZZ_2$ be arbitrary.
It is well-known that $f'$ extends to a (unique) $\ZZ_2$-flow $f$ in $G$.  If $f'$ is constantly equal to $1$, we say
that $f$ is the \emph{canonical $\ZZ_2$-flow with respect to~$T$}.  Note that in such case, $f$ is nowhere-zero
except for the edges of $T$.

\begin{lemma}\label{lemma-noncannon}
Let $G$ be an $n$-vertex graph. If $G$~has a spanning tree~$T$
such that $G-E(T)$ is connected and at least $q$ edges of\/ $T$ are assigned value $1$ in the canonical $\ZZ_2$-flow with respect to~$T$,
then $G$ has at least $2^q$ nowhere-zero $\ZZ_2^2$-flows.
\end{lemma}

\begin{proof}
Since $G-E(T)$ is connected, $G$ contains a spanning tree $T'$ disjoint from~$T$.
Let $f$ be the canonical $\ZZ_2$-flow with respect to~$T$, and let $X$ be the set of edges of~$T$
that are assigned value $1$ in the canonical $\ZZ_2$-flow with respect to~$T$.
Note that any assignment of elements of $\ZZ_2$ to $X$ extends to a $\ZZ_2$-flow~$g$ on~$G$ that may be zero only on~$X\cup E(T')$,
where $f$~is non-zero. Thus, for each such~$g$, the pair $(f,g)$ is a nowhere-zero $\ZZ_2\times\ZZ_2$-flow.
This way, we obtain $2^{|X|}=2^q$ nowhere-zero $\ZZ_2\times \ZZ_2$-flows in $G$.
\end{proof}

Let $T_1$ and $T_2$ be disjoint spanning trees of a graph, and suppose that a vertex~$v$
has degree two both in $T_1$ and $T_2$. Let $e_1,e'_1\in E(T_1)$ and $e_2,e'_2\in E(T_2)$ be the edges of $T_1$ and $T_2$
incident with $v$.  If for $i=1$ and for $i=2$, the ends of $e_{3-i}$ and $e'_{3-i}$ different from $v$ are in
different components of $T_i-v$, then let $(T'_1,T'_2)=(T_1-e_1-e'_1+e_2+e'_2,T_2-e_2-e'_2+e_1+e'_1)$.  Otherwise,
we may assume that the ends of $e_1$, $e_2$, and $e_2'$ are in the same component of $T_1-v$.  By symmetry between $e_2$ and $e_2'$,
we can assume that the ends of $e_1$ and $e_2$ are in the same component of $T_2-v$.  In this case we define $(T'_1,T'_2)=(T_1-e_1+e_2,T_2-e_2+e_1)$.
Observe that in both cases, $T'_1$ and $T'_2$ are trees; we say that the pair $(T'_1,T'_2)$ is obtained from $(T_1,T_2)$ by
a \emph{flip at $v$}.

Suppose now that $v$ is a leaf of $T_1$, with incident edge $e_1\in T_1$, and let $u$ be the other vertex incident with $e_1$.
Let $e_2\in E(T_2)$ be the edge joining $v$ to the unique component of $T_2-v$ containing $u$.
Let $(T'_1,T'_2)=(T_1-e_1+e_2, T_2-e_2+e_1)$.  Again, $T'_1$ and $T'_2$ are trees; we say that the pair $(T'_1,T'_2)$ is obtained from $(T_1,T_2)$ by a \emph{flip} at $v$.

\begin{lemma}\label{lemma-exppairs}
If an $n$-vertex graph $G$~has two disjoint spanning trees, then $G$~has at least $2^{n/12}$ pairs of disjoint spanning trees.
\end{lemma}

\begin{proof}
Let $T_1$ and $T_2$ be disjoint spanning trees in $G$. We may assume that $G=T_1\cup T_2$.
For $i\in\{1,2\}$, let $L_i$ be the set of leaves of $T_i$ and let $n_i=|L_i|$ be the number of leaves of $T_i$.
Let $V_4$ be the set of vertices of $G$ that have degree two both in $T_1$ and $T_2$, and let $n_4=|V_4|$.

Since $G$ is union of two trees, every subgraph of $G$ contains a vertex of degree at most 3. This implies that
$G$ is $4$-colorable. Hence, there exists an independent set $X\subseteq L_1\cup L_2\cup V_4$ of size at least $|L_1\cup L_2\cup V_4|/4$.
Let $X'$ be an arbitrary subset of $X$.
For all $v\in X'$ in an arbitrary order, we transform the current pair of spanning trees by replacing it with the pair obtained by a flip at $v$.
By performing this procedure, we obtain for each $X'\subseteq X$ a different pair of disjoint spanning trees
of $G$; hence, there are at least $2^{|X|}\ge 2^{|L_1\cup L_2\cup V_4|/4}$ such pairs.
Note that for $i=1,2$, the tree $T_i$ has less than $n_i$ vertices of degree greater than two.  Therefore, $|L_1\cup L_2\cup V_4| \ge n-n_1-n_2$. On the other hand, 
$|L_1\cup L_2| \ge \tfrac{1}{2}(n_1+n_2)$. Thus, there are at least $2^t$ pairs of disjoint spanning trees, where 
$t = \lceil \tfrac{1}{4} \max\{n-n_1-n_2,\tfrac{1}{2}(n_1+n_2)\}\rceil \ge \tfrac{1}{12}n$.
\end{proof}

\begin{theorem}\label{thm-4flows}
If an $n$-vertex graph $G$ with $m$ edges has two disjoint spanning trees, then the following holds:
\begin{itemize}
  \item[\rm (a)]
  $G$ has at least $2^{n/250}$ nowhere-zero $\ZZ_2\times \ZZ_2$-flows.
  \item[\rm (b)]
  $G$ has at least $3^{m-2n+2}$ nowhere-zero $\ZZ_2\times \ZZ_2$-flows.
\end{itemize}
\end{theorem}

\begin{proof}
Let $T_1$ and $T_2$ be disjoint spanning trees of $G$.
Statement (b) is straightforward: first we give values $(1,0)$ to the edges in $E(T_1)$, values $(0,1)$ on $E(T_2)$. Each edge in the complement of both spanning
trees can be given flow value $(1,0)$, $(0,1)$ or $(1,1)$ arbitrarily.
Then we extend the first coordinate value to a $\ZZ_2$-flow on~$G$ by modifying appropriately values on~$E(T_2)$, and similarly for the second coordinate and $E(T_1)$.
The combined $\ZZ_2\times \ZZ_2$-flow is always nowhere-zero, so this gives $3^{m-2n+2}$ different nowhere-zero
$\ZZ_2\times \ZZ_2$-flows in $G$.

In order to prove (a), we may assume that $m\le 2.005n-2$ since (b) can be applied otherwise.
Let $q$ be the maximum integer for which $G$ has a spanning tree $T$ such that $G-E(T)$ is connected and
$q$ edges of $T$ are assigned value $1$ in the canonical $\ZZ_2$-flow with respect to $T$.  If $q\ge n/250$, then
$G$ has at least $2^{n/250}$ nowhere-zero $\ZZ_2\times \ZZ_2$-flows by Lemma~\ref{lemma-noncannon}; hence, assume that $q < q_0 \colonequals n/250$.

Let $\TT$ be the set of all pairs of disjoint spanning trees of $G$; by Lemma~\ref{lemma-exppairs}, we have $|\TT|\ge 2^{n/12}$.
For each pair $(T,T')\in \TT$, let $f_{T,T'}$ be the nowhere-zero $\ZZ_2\times \ZZ_2$-flow in $G$ consisting of the canonical
$\ZZ_2$-flows with respect to~$T$ and~$T'$.

For a $\ZZ_2\times \ZZ_2$-flow~$(\phi_1,\phi_2)$ let us estimate $X(\phi_1,\phi_2)$, the number of pairs $(T_1,T_2)$ of trees with
$f_{T_1,T_2} = (\phi_1,\phi_2)$. For $i=1,2$, each tree~$T_i$ of such a pair must contain edges of the set
$Z_i = \{e : \phi_i(e) = 0\}$.  Since $|Z_i|\ge n-1-q$, there are at most
$
   \binom{m-|Z_i|}{n-1-|Z_i|} \le \binom{m-(n-1-q)}{q}
$
ways to choose~$T_i$ such that its canonical flow is~$\phi_i$.
It follows that $X(\phi_1, \phi_2)$ is at most
$$
   \binom{m-(n-1-q)}{q}^2\le \binom{1.01n}{q}^2 \le \left(\frac{1.01en}{q}\right)^{2q}\le \left(\frac{1.01en}{q_0}\right)^{2q_0}
   < 700^{n/125}.
$$
In the first inequality above we have used the bound on $q$ and on the number of edges of $G$, and in the second one we used the fact that
$\binom{x}{q} \le \bigl(\frac{ex}{q}\bigr)^{q}$ (where $e$ is the base of the natural logarithm).
Consequently, the number of nowhere-zero $\ZZ_2\times \ZZ_2$-flows $\{f_{T,T'}:(T,T')\in \TT\}$ is at least
$$
  \frac{2^{n/12}}{700^{n/125}} = 2^{n/12-\log_2(700)n/125}\ge 2^{n/250}.
$$
\end{proof}

As every $4$-edge-connected graph has two disjoint spanning trees, we get the following consequence.

\begin{corollary}
If $G$ is a $4$-edge-connected graph with $n$ vertices, then $G$ has at least\/ $2^{n/250}$ nowhere-zero $\ZZ_2^2$-flows.
\end{corollary}

\section{$\ZZ_3$-flows}\label{sec:3flows}

We view $\ZZ_3$-flows as orientations of the graph such that each vertex has the same indegree and outdegree modulo $3$.
We will use a result of L.\,M.\,Lov\'asz et al., who proved in~\cite{ltwz} that every 6-edge-connected graph admits a nowhere-zero $\ZZ_3$-flow.
To state their result precisely, we need to introduce
a generalization of flows on graphs and a delicate way to measure graph connectivity.

For a graph $G$, a function $\beta:V(G)\to \ZZ_3$ is a \emph{boundary} if the sum of its values is $0$.
A \emph{nowhere-zero $\beta$-flow} in a graph $G$ is an orientation of $G$ such that every $v\in V(G)$
satisfies $\deg^+(v)-\deg^-(v)\equiv \beta(v)\pmod 3$. In particular, if $\beta$~is identically~$0$, then
a nowhere-zero $\beta$-flow defines a nowhere-zero $\ZZ_3$-flow.  For a set $X\subseteq V(G)$, let $\beta(X)=\sum_{x\in X}\beta(x)$
and let $\deg(X)$ denote the number of edges of $G$ with exactly one end in $X$.  Let $\sigma(X)$ be defined as follows:

$$
  \sigma(X)=\begin{cases}
	4&\text{if $\beta(X)=0$ and $\deg(X)$ is even}\\
	7&\text{if $\beta(X)=0$ and $\deg(X)$ is odd}\\
	6&\text{if $\beta(X)\neq 0$ and $\deg(X)$ is even}\\
	5&\text{if $\beta(X)\neq 0$ and $\deg(X)$ is odd.}
  \end{cases}
$$
Note that we use $\sigma$ where the authors of~\cite{ltwz} write $4 + |\tau|$. 
The next result appears as Theorem~3.1 in~\cite{ltwz} (in a slightly stronger formulation, to aid the inductive proof there). 

\begin{theorem}[Lov{\'a}sz et al.~\cite{ltwz}]\label{thm-extend}
  Let $G$ be a graph, let $\beta$ be its boundary, and let $v$ be a vertex of $G$.
  Suppose that
  \begin{enumerate}
    \item $\deg(X)\ge\sigma(X)$ for all sets $X\subsetneq V(G)$ such that $v \in X$ and $|X| \ge 2$,
    \item $\deg(v)\le\sigma(\{v\})$.
  \end{enumerate}
  Then every orientation of edges incident with $v$ such that
  $\deg^+(v)-\deg^-(v)\equiv \beta(v)\pmod 3$ extends to a nowhere-zero $\beta$-flow in $G$.
\end{theorem}

We will in fact use the following corollary: 

\begin{corollary}\label{cor:extend} 
  Let $G$ be a 6-edge-connected graph, let $\beta$ be its boundary, and let $v$ be a vertex of $G$.
  Suppose $\deg(v) \le 7$ and $\beta(v) = 0$. 
  Then every orientation of edges incident with $v$ such that
  $\deg^+(v)-\deg^-(v)\equiv \beta(v)\pmod 3$ extends to a nowhere-zero $\beta$-flow in $G$.
\end{corollary}

\begin{proof}
If $\deg_G(v) = 7$, we put $G' = G$ and $\beta' = \beta$. 
Otherwise (i.e., if $\deg_G(v) = 6$), we pick an edge~$vw$ in~$G$ (the case of an edge directed towards~$v$ is symmetric). 
Then we put $G' = G-vw+wv$, and define $\beta'(u) = \beta(u)$ for $u\ne v,w$, $\beta'(v) = \beta(v) - 2$, and 
$\beta'(w) = \beta(w) + 2$. 
With this choice, any nowhere-zero $\beta'$-flow in~$G'$ yields (after we change~$wv$ back to~$vw$) a nowhere-zero $\beta$-flow in~$G$. 

We will argue that a nowhere-zero $\beta'$-flow in $G'$ exists by Theorem~\ref{thm-extend}. 
Condition~1 follows from 6-edge-connectivity of~$G$. 
Condition~2 is true for both choices of~$\deg(v)$ (for this we needed to reverse an edge incident to~$v$ if $\deg(v) = 6$). 
\end{proof}

We say a graph is minimally $k$-edge-connected, if it is $k$-edge-connected, but deletion of any edge
creates a cut with $k-1$ edges. Mader~\cite{Mader_1971,Mader_1974} proved that such graph must have
many vertices of degree~$k$; we will use an improved bound due to Cai~\cite{Cai_1993}. 
We only state the result for case $k=6$, but there are versions for other values of~$k$. 

\begin{theorem}[Cai \cite{Cai_1993}] \label{thm:Cai}
  Suppose $G$ is a minimally $6$-edge-connected simple graph. 
  Then $G$ has at least $\tfrac{11}{30}|G| + \tfrac{85}{30}$ 
  vertices of degree~$6$. 
\end{theorem}

The above theorem does not hold for a graph with multiple edges. To be able to use it for such graphs, we will use a simple 
construction to get rid of parallel edges. Let $u$ be a vertex of an undirected graph~$G$. 
We delete possible loops at~$u$, we subdivide every edge incident to~$u$ and we let
$X$~be the set of the \emph{new vertices}. Next, we put a clique on~$X$ and delete~$u$. 
We will say that the resulting graph~$G'$ was obtained from~$G$ by \emph{clique expansion} with \emph{center~$u$}. 
The following easy lemma will be crucial. 

\begin{lemma}\label{l:clique_expansion}
Let $u$ be a vertex of a $k$-edge-connected graph~$G$. Suppose $|V(G)|\ge 2$ and let $G'$ be obtained by clique expansion with center~$u$ as described above. 
Then $G'$ is also $k$-edge-connected. 
\end{lemma}

\begin{proof}
  Consider an edge-cut $\delta(A)$ for some $A \subset V(G')$. If all new vertices are in~$A$ (or none of them is) then we have a cut of the same size in~$G$. 
  Otherwise, we may decrease the number of edges in~$\delta(A)$ by moving all of the new vertices to~$A$, or by moving all of them out of~$A$. 
\end{proof}

In a $6$-edge-connected graph, a pair $e,e'$ of edges incident with the same vertex $u$ will be called \emph{$6$-splittable at~$u$}
if the graph obtained by lifting $e$ and $e'$ from~$u$ has no edge cuts of size less than~$6$ possibly with the exception of the cut defined by~$u$.

\begin{theorem}
\label{thm:3flowexpmany}
  Every $6$-edge-connected graph $G$ with $n \ge 2$ vertices has at least~$2^{(n-2)/12}$ nowhere-zero $\ZZ_3$-flows. 
In fact, if $\Delta(G) \le 7$, then every orientation of edges incident with a given vertex $v \in V(G)$ such that $\deg^+(v)-\deg^-(v)\equiv 0\pmod 3$
extends to at least $2^{(n-2)/12}$ nowhere-zero $\ZZ_3$-flows in $G$. 
\end{theorem}

\begin{proof}
Put $g(n) = 2^{(n-2)/12}$. 
The first statement follows from the second one: for a general $6$-edge-connected graph we may start by 
a repeated application of Corollary~\ref{cor-mader2}, to ensure that $\Delta(G) \le 7$. 

We will prove the second statement by induction on~$n + |E(G)|$. If $n=2$ then there is at least one such extension 
  (more if we have loops on the other vertex). Next, suppose $n > 2$. 
We start by discussing several possible arguments for the induction step; at the end we argue that at least one of them applies. 

\paragraph{Case 1. Three ways to split.} Suppose, there is a vertex $s$ of degree 6 and edges $e_1$, $e_2$, $e_3$ incident to~$s$ such 
  that all three pairs $e_i$, $e_j$ are $6$-splittable at~$s$. Suppose further that $s \ne v$. 
  Let $G_{ij}$ for $1\le i<j\le 3$ be the graph obtained from $G$ by lifting the pair $e_i$ and $e_j$ at $s$,
  and then lifting another pair at $s$ chosen by Lemma~\ref{lemma-mader} and suppressing $s$, so that $G_{ij}$ is $6$-edge-connected with 
  one vertex less. If $s$ and $v$ are adjacent (possibly by several edges), we must be somewhat careful: after lifting a pair of edges including 
  an edge incident with~$v$, we orient the new edge consistently. The possible danger -- lifting edges $e$, $e'$ with the same ends $s$ and $v$ -- 
  cannot occur, because $\deg(v) \le 7$. 

  By the induction hypothesis, $G_{ij}$ has $n_{ij} \ge g(n-1)$ nowhere-zero $\ZZ_3$-flows extending the orientation of $\delta(v)$. 
  Each such flow $f_{ij}$ naturally corresponds to a nowhere-zero $\ZZ_3$-flow $f'_{ij}$ in $G$.  Suppose that a nowhere-zero $\ZZ_3$-flow $f$ of $G$ is equal
  to $f'_{12}$ for some nowhere-zero $\ZZ_3$-flow $f_{12}$ of $G_{12}$, as well as to $f'_{13}$ for some nowhere-zero
  $\ZZ_3$-flow $f_{13}$ of $G_{13}$.  By symmetry, we can in this case assume that $f$ orients $e_1$ towards~$v$
  and $e_2$ and $e_3$ away from~$v$.  However, this implies that there exists no nowhere-zero $\ZZ_3$-flow $f_{23}$ of $G_{23}$
  with $f=f'_{23}$.  Consequently, $G$ has at least $\frac{1}{2}(n_{12}+n_{13}+n_{23})\ge \frac 32 g(n-1) \ge g(n)$ nowhere-zero $\ZZ_3$-flows extending 
  (the orientation of) $\delta(v)$. 

\paragraph{Case 2. Small cut.} Suppose there is $Y \subseteq V(G)$ such that $\deg(Y) \le 7$ and $|Y|, |V(G)\setminus Y| \ge 2$. 
  We may assume $v \not\in Y$. 
  Let $G'$ be the graph obtained from~$G$ by identifying all vertices of $Y$ to a single vertex $y$ and removing all loops
  at $y$. Note that $G'$ has $n' = |V(G)\setminus Y|+1 \ge 2$ vertices, it is 6-edge-connected, and $n' < n$. 
  Thus, we may extend $\delta(v)$ to at least $g(n')$ nowhere-zero $\ZZ_3$-flows in~$G'$. 
  Next we let $G''$ be the graph obtained by identifying the complement of~$Y$ to a single vertex~$y''$. 
  Clearly, $G''$ has $n'' = |Y| + 1 \ge 2$ vertices, and it is 6-edge-connected. Moreover, $n'' < n$ and $n' + n'' = n + 2$. 
  For each $\ZZ_3$-flow we found in~$G'$ we orient $\delta(y'')$ in~$G''$ according to the orientation of~$\delta(y')$ in~$G'$. 
  Using induction assumption again, we get at least $g(n'')$ nowhere-zero $\ZZ_3$-flows in~$G''$ extending $\delta(y'')$, getting 
  together at least $g(n')\cdot g(n'') = g(n)$ nowhere-zero $\ZZ_3$-flows in~$G$ extending~$\delta(v)$. 

\paragraph{Case 3. $\mathbf{n \le 14}$} 
  We need just $g(n) \le 2$ extensions of~$\delta(v)$. 
  If $n=3$, then by our connectivity and degree assumption $G$ is a triangle with every edge of multiplicity 3 or 4. It is easy to check that 
  this graph has at least $2$ extensions of~$\delta(v)$. 
  In all other cases, there is a vertex $w \in N(v)$ and an edge $e \not\ni v,w$. 
  We start by choosing an orientation of~$e$, we want to extend both of them together with the given orientation of~$\delta(v)$ to the whole of~$G$.
  To this end we let $f$ be the edge $vw$ or $wv$ (or one of these edges, if there are several) and we put $G' = G - \{e,f\}$. 
  Next, define $\beta : V(G') \to \ZZ_3$ so that any $\beta$-flow on~$G'$ corresponds to a nowhere-zero $\ZZ_3$-flow on~$G$. (There is a unique such mapping~$\beta$, it is
  equal to~$\pm 1$ at vertices incident to~$e$ or~$f$ and to zero elsewhere.) We verify assumptions of Theorem~\ref{thm-extend}: In $G'$ we have $\beta(v) \ne 0$ and
  $\deg(v) \in \{5,6\}$, thus condition~2 is satisfied. Regarding condition~1, for every~$X$ as in the condition, we have $\deg_{G'}(X) \ge \deg_G(X)-2$. 
  We know that $|X|\ge 2$, assume also $|V(G)\setminus X| \ge 2$. If $\deg_G(X) \le 7$, we may use Case~2. Otherwise $\deg_{G'}(X) \ge 6$ and we have $\deg_{G'} \ge \sigma(X)$. 
  It remains to check the case when the complement of~$X$ is a single vertex.
  In this case either $\deg_{G'}(X) = \deg_G(X)-1 \ge 5$ and $\beta(X)\ne 0$ or 
  $\deg_{G'}(X) = \deg_G(X) \ge 6$ and $\beta(X) = 0$. In both of these cases, condition~1 is satisfied. 

\paragraph{Case 4. $G$ has a vertex~$s \ne v$ of degree~6} 
  If $n = 3$ we use Case 3; if $n>3$ and $s$ is connected to some~$x$ by at least three edges, we use Case~2 with $Y=\{x,s\}$. 
  Thus $s$ has at least three distinct neighbors $v_1$, $v_2$, $v_3$. Let $e_i$ be an edge connecting $s$ to~$v_i$. If Case~1 does not apply, then one of the liftings is
  not possible; we may assume $e_1$, $e_2$ is not a splittable pair. This means that there is a set $Y \subseteq V(G)$ such that $v_1, v_2 \in Y$, $s \notin Y$
  and $\deg(Y) \le 7$. The complement of $Y$ cannot be just~$\{s\}$, otherwise $Y$ is not a bad cut after the lifting; 
  moreover $|Y| \ge 2$. Thus, we can apply Case~2 argument to finish the proof.

\paragraph{Case 5. $G$ has a double edge not incident to~$v$} 
  Suppose vertices $x \ne v$ and $y \ne v$ of~$G$ are connected by at least two edges, $e$ and~$f$. 
  If $G-e$ is 6-edge-connected, we may find by induction assumption $g(n)$ flows in it that extend the given preorientation of~$\delta(v)$. 
  Each of these can be modified to a flow in~$G$: we change direction of~$f$ and return~$e$ with the same direction as~$f$. 
  So $G$ has an edge-cut of size~6 containing~$e$. Thus we are in Case~2, unless this cut is trivial. 
  This means that $x$ or~$y$ is a vertex of degree~6, and we are done by Case~4.

\bigskip

To summarize, if we cannot apply any of the induction steps above, all parallel edges in~$G$ are incident with~$v$ and $n \ge 15$. 
Next, we construct a simple graph~$G'$ by forgetting the preorientation of~$G$ and applying the clique expansion with center~$v$; 
let $X$ be the set of the new vertices. 
By Lemma~\ref{l:clique_expansion}, graph $G'$~is 6-edge-connected.  
Now we choose any maximal set $F' \subseteq E(G')$ such that $G'-F'$ is 6-edge-connected.  
Letting $F$ be the edges of~$F'$ that are not incident with~$X$, graph $G-F$ is 6-edge-connected as well. 
Note that if $|X| = 6$, the new vertices have all degree~$6$ and so $F = F'$. 
If $|X| = 7$, there is at most one edge of~$F'$ in~$\delta(X)$ and at most one incident with each vertex of~$X$, thus 
$|F| \ge |F'| - 4$. 

\paragraph{Case A: $|F| \ge (n-2)/12$} 

We may choose arbitrary orientation of edges in~$F$, define boundary $\beta$ by putting 
$\beta(u) = \deg^-_F(u) - \deg^+_F(u)$. 
By Corollary~\ref{cor:extend}, there is an extension of $\delta(v)$ to a $\beta$-flow on~$G-F$. 
Combining with the chosen orientation of~$F$ we get a nowhere-zero $\ZZ_3$-flow on~$G$ extending $\delta(v)$, 
altogether at least $2^{|F|} \ge g(n)$ such flows. 

\paragraph{Case B: $|F| < (n-2)/12$} 

Maximality of~$F'$ implies that $G'-F'$ is a minimally 6-edge-connected simple graph. 
By Theorem~\ref{thm:Cai} there are at least $t = \frac{11}{30}|G'| + \frac{85}{30}$ vertices of degree~$6$ in~$G'-F'$. 
Our aim is to get a vertex of degree~6 in~$V(G) \setminus \{v\}$ and finish the proof by Case~4. 
To this end, we will find a vertex of degree~$6$ in $G'-F'$ that is not incident with an edge in~$F'$ and is 
not one of the new vertices. 
If $|X| = 6$ we require $\frac{11}{30}(n+5) + \frac{85}{30} - 2|F| - 6 > 0$. If $|X|=7$ we need 
somewhat stronger $\frac{11}{30}(n+6) + \frac{85}{30} - 2(|F| + 4) > 0$. 
Both of these are valid for $n \ge 15$, which finishes the proof.
\end{proof}

We remark that an earlier version of the paper used a different way to finish the proof. 
Either a graph has a vertex of degree~6, or it has many edges. In the latter case, we applied Theorem~\ref{thm-extend} 
and a probabilistic argument to get $2^{\frac{(m-3n-1)^2}{384n}}$ nowhere-zero $\ZZ_3$-flows. 
In comparison with our present approach this was more technical and yielded a worse final result. 
We are grateful to the referee for suggesting the approach using minimal 6-edge-connected graphs.

\section*{Future work}

A natural task is to improve our bounds. Specifically, what are the largest constants $c_3$, $c_4$, $c_6$ such that
\begin{itemize}
  \item every $3$-edge-connected graph has at least $c_6^n$ nowhere-zero $\ZZ_2\times \ZZ_3$-flows?
  \item every $4$-edge-connected graph has at least $c_4^n$ nowhere-zero $\ZZ_2^2$-flows?
  \item every $6$-edge-connected graph has at least $c_3^n$ nowhere-zero $\ZZ_3$-flows?
\end{itemize}

Our argument for $\ZZ_3$-flows breaks for $\beta$-flows (one of the key steps is splitting a vertex). We would like to know, 
if an analogy of Theorem~\ref{thm:3flowexpmany} holds for $\beta$-flows. 

In another direction, it would be interesting to find whether there is a natural setting, where the number of nowhere-zero flows
is larger than a constant but smaller than an exponential.

\section*{Acknowledgements}
The motivation for this research was sparked during the inspiring workshop New Trends in Graph Coloring at BIRS (Banff, Alberta) in 2016.
We thank the anonymous referee for suggestions leading to a simplification of the proof of Theorem~\ref{thm:3flowexpmany} and 
to an improvement of the bound.

\bibliographystyle{acm}
\bibliography{flows}

\end{document}